\documentclass[11pt,dvipsnames,reqno]{amsart}
\usepackage[utf8]{inputenc}
\usepackage[T1]{fontenc}
\usepackage{enumitem}
\usepackage[margin=2.5cm]{geometry}

\let\origleft\left
\let\origright\right
\renewcommand\left{\mathopen{}\mathclose\bgroup\origleft}
\renewcommand\right{\aftergroup\egroup\origright}

\usepackage[justification=centering]{caption}
\usepackage{subcaption}
\usepackage{tikz}
\usetikzlibrary{arrows.meta,decorations.markings,lindenmayersystems}
\tikzset{->-/.style={decoration={
  markings,
  mark=at position #1 with {\arrow{>}}},postaction={decorate}}}

\usepackage{thmtools}
\declaretheorem{theorem}
\declaretheorem[numberwithin=section]{lemma}
\declaretheorem{proposition, corollary, remark, example}[
  style=plain,
  sibling=lemma
]

\usepackage[allcolors=BrickRed,colorlinks]{hyperref}
\usepackage[capitalize,nameinlink]{cleveref}
\crefdefaultlabelformat{#2{\scshape #1}#3}
\crefformat{equation}{(#2#1#3)}
\crefmultiformat{equation}{(#2#1#3)}%
{ and~(#2#1#3)}{, (#2#1#3)}{ and~(#2#1#3)}
\crefrangeformat{equation}{(#3#1#4) to~(#5#2#6)}
\crefname{figure}{Figure}{Figures}

\usepackage{amssymb}
\usepackage{dsfont}
\usepackage{esvect}
\usepackage{mathtools}
\usepackage{suffix}
\mathtoolsset{centercolon}

\newcommand\shortcut[3]{%
  \expandafter\xdef\csname #1#2\endcsname{\noexpand#3{#2}}%
}
\foreach \x in {A,...,Z} {%
  \shortcut{b}{\x}{\mathbf}
  \shortcut{c}{\x}{\mathcal}
  \shortcut{d}{\x}{\mathbb}
  \shortcut{f}{\x}{\mathfrak}
}
\let\eps\varepsilon
\newcommand\dd{\mathrm d}
\DeclareMathOperator\II{\mathds1}
\DeclareMathOperator\EE\dE
\DeclareMathOperator\PP\dP
\DeclareMathOperator\Tr{Tr}

\newcommand\RR{\dR}
\newcommand\CC{\dC}
\newcommand\MM{M_d}
\newcommand\rv{o}
\newcommand\RT[1][d]{\dT_{#1}}
\WithSuffix\newcommand\RT>[1][d]{T_{#1}}
\newcommand\RTA[1][d]{\dA_{#1}}
\WithSuffix\newcommand\RTA>[1][d]{A_{#1}}
\newcommand\RG[1][d,n]{\dG_{#1}}
\WithSuffix\newcommand\RG>[1][d,n]{G_{#1}}
\newcommand\RGA[1][d,n]{\dA_{#1}}
\WithSuffix\newcommand\RGA>[1][d,n]{A_{#1}}
\newcommand\NC{\mathrm{NC}}
\newcommand\ANC{\mathrm{ANC}}
\newcommand\gen{e}
\newcommand\KM{\mu_{\RT}}
\newcommand\CM{\mathbf{CM}_{d,n}}
\newcommand\CMA{\mathbf{A}_{d,n}}
\newcommand\Simple{{\cS_{d,n}}}
\WithSuffix\newcommand\KM>{\mu_{\RT>}}
\newcommand\conv[1][n]{\mathop{\;\xrightarrow[#1\to\infty]{}\;}}
\newcommand\Dash{\mathrel{\rule[2.5pt]{1.5em}{.5pt}}}
\newcommand\ew{\varnothing}

\newcommand\ep{\emptyset}
\newcommand\finer{\preceq}
\WithSuffix\newcommand\finer*{\prec}
\newcommand\GP{\cP}
\newcommand\skel[1][w]{\sigma_{#1}}
\newcommand\bad{\lhd}

\newcommand\mail[1]{\normalfont\href{mailto:#1}{\texttt{#1}}}

\title[Star moments of regular directed graphs and trees]{%
  A combinatorial view on star moments\\ 
  of regular directed graphs and trees}

\author[B.~Dadoun]{Benjamin Dadoun}
\address[BD]{Mathematics, Division of Science, New York University Abu Dhabi, UAE}%
\email{\mail{benjamin.dadoun@gmail.com}}%

\author[P.~O.~Santos]{Patrick Oliveira Santos}
\address[POS]{Univ Gustave Eiffel, Univ Paris Est Creteil, CNRS, LAMA UMR8050, F-77447 Marne-la-Vallée, France}
\email{\mail{patrick.oliveirasantos@u-pem.fr}}%

\begin{document}

\begin{abstract}
    We investigate the method of moments for
    $d$-regular digraphs and the limiting $d$-regular directed tree~$\RT>$
    as the number of vertices tends to infinity, in the same spirit as
    McKay~\cite{mckay1981expected} for the undirected setting.
    In particular, we provide
    a combinatorial derivation of the formula for the star moments
    (from a root vertex $\rv\in\RT>$)
    \[
      \MM(w)\qquad:=\sum_{\substack{v_0,v_1\ldots,v_{k-1},v_k\in\RT>\\v_0=v_k=\rv}}
      \RTA>^{w_1}(v_0,v_1)\RTA>^{w_2}(v_1,v_2)%
      \cdots%
      \RTA>^{w_k}(v_{k-1},v_k)
    \]
    with~$\RTA>$ the adjacency matrix of~$\RT>$,
    where $w:=w_1\cdots w_k$ is any word on the alphabet $\{1,{*}\}$
    and~$\RTA>^*$ is the adjoint matrix of~$\RTA>$.
    Our analysis highlights a connection between the
    non-zero summands of~$\MM(w)$ and the non-crossing partitions
    of $\{1,\ldots,k\}$ which are in some sense compatible with~$w$.
\end{abstract}

\maketitle
\thispagestyle{empty}

\section{Introduction}
Counting paths in graphs and other discrete structures
is a standard question with applications to
many areas of mathematics, see~%
\cite{kesten1959symmetric,bender1978asymptotic,wormald1980some,woess2000random}
or the book~\cite{bollobas1998random}.
In random matrix theory, this question is typically raised
when studying the convergence of empirical spectral distributions
(ESDs)
through the method of moments,
of which Wigner's original proof of the semicircular law~\cite{wigner1958}
is a renowned example. Essentially,
for random Hermitian matrices
$W_n:=(W_n(i,j))_{1\le i,j\le n}$
whose coefficients on and above the diagonal
are i.i.d.\ with mean~$0$ and variance~$1$
(so called Wigner matrices), the different summands
$\EE{W_n(i_1,i_2)\cdots W_n(i_k,i_1)}$
occurring in the expansion of the states
$\EE{\Tr W_n^k},\,k\ge1,$
can be related to certain cycles $i_1\to\cdots\to i_k\to i_1$
in a graph with vertex set $[n]:=\{1,\ldots,n\}$,
and understanding the combinatorics of these cycles
helps to determine how each of those summands contributes
to the $k$-th moment of the limiting spectral distribution
(the semicircle distribution).

In contrast, when~$\RGA\in{\{0,1\}}^{n\times n}$
is the adjacency matrix of a uniformly sampled
graph~$\RG$ with~$n$ vertices and constant degree~$d\ge2$
(that is, $\RG$ is a uniform $d$-regular graph on~$[n]$
and~$\RGA(i,j)=1$ if and only if $\{i,j\}$ is an edge
in~$\RG$),
McKay~\cite{mckay1981expected}
showed using the same method that the mean
ESD $\EE{\frac1n\sum_{i=1}^n\delta_{\lambda_i(\RGA)}}$
associated with~$\RGA$'s eigenvalues
$\lambda_1(\RGA),\ldots,\lambda_n(\RGA)$
converges weakly (and in moments)
towards a certain probability measure~$\KM$
which is now known as the Kesten--McKay distribution,
in the sense that
\begin{equation}
  \frac1n\EE{\sum_{i=1}^nf\bigl(\lambda_i(\RGA)\bigr)}
  \conv
  \int f(x)\,\KM(\dd x)\label{eq:km}
\end{equation}
holds for any polynomial or continuous bounded function~$f:\RR\to\RR$.
When $f(x):=x^k$ for some positive integer~$k$,
the left-hand side of~\eqref{eq:km},
which can also be written $\frac1n\EE{\Tr{\RGA^k}}$, coincides with
the expected total number of excursions of length~$k$
from a uniformly chosen vertex in~$\RG$, while the
right-hand side of~\eqref{eq:km} counts the number
of such excursions (from a fixed vertex)
in the (infinite) $d$-regular tree~$\RT$,
which is the Cayley graph of the free group with
presentation $\langle\gen_1,\ldots,\gen_d\mid\gen_i^2=1\rangle$,
see \cref{fig:3regtree}.
In fact, since the graphs~$\RG$ converge locally
to the tree~$\RT$ as $n\to\infty$ (i.e.,
with respect to the Benjamini--Schramm
topology~\cite{benjamini2011recurrence}),
the convergence~\eqref{eq:km} of their mean ESDs can be recovered
from Bordenave--Lelarge's criterion~\cite{bordenave2010resolvent}.

% Figures fig:3regulartree, fig:d2regulatree
\begin{figure}
\centering% 
\begin{subfigure}[c]{.5\textwidth}\centering%
\includegraphics{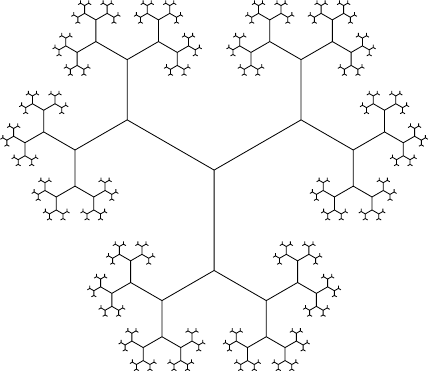}
\subcaption{}\label{fig:3regtree}
\end{subfigure}\hfill
\begin{subfigure}[c]{.5\textwidth}\centering%
\includegraphics{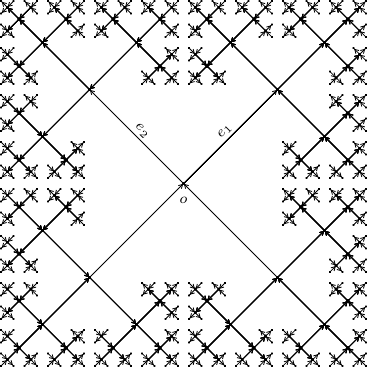}
\subcaption{}\label{fig:d2regtree}
\end{subfigure}%
\caption{(\textsc{a}) The undirected tree~$\RT[3]$
and (\textsc{b}) the directed tree~$\RT>[2]$.}
\end{figure}

In the present note, we adapt McKay's approach
to the asymmetric (i.e., oriented) case.
Although the local convergence as $n\to\infty$
of uniform $d$-regular digraphs~$\RG>$
towards the $d$-regular directed tree~$\RT>$
does hold in a similar fashion
(w.r.t.\ the ``oriented'' Benjamini--Schramm topology),
it does not imply the convergence of ESDs anymore,
and the analogue of~\eqref{eq:km} for oriented regular graphs
is still an open question, known as the
\emph{oriented Kesten--McKay conjecture}~\cite{bordenave2012around}:
the ESD of~$\RG>$ should converge towards
a probability distribution
on~$\CC$
corresponding in some sense to the spectral measure
of~$\RT>$.
One difficulty for this conjecture is that, because the
adjacency matrix~$\RGA>$ of~$\RG>$ is no longer
Hermitian, the tracial moments $\EE{\Tr{\RGA>^k}},\,k\ge0,$
do not continuously determine the (now complex-valued)
mean ESD of~$\RGA>$. In fact,
as~$\RGA>$ is not even normal, neither do the
star moments $\EE{\Tr{\RGA>^w}}$
defined for any bit string $w:=w_1\cdots w_k$
on the alphabet $\Sigma:=\{1,{*}\}$,
where $\RGA>^w:=\RGA>^{w_1}\cdots \RGA>^{w_k}$,
and~$\RGA>^*$ stands for the adjoint matrix
of~$\RGA>=:\RGA>^1$ (said differently,
$\RGA>^*$ is the adjacency matrix
of the graph~$\RG>^*$ obtained from~$\RG>$
by reversing each of its arcs).
Nonetheless, investigating the star moments of
regular digraphs remains interesting
from a combinatorial perspective,
and their convergence towards
the corresponding star moments of the
regular directed tree suggests that the conjecture holds true.

By definition, the $d$-regular directed tree~$\RT>$
is the unique infinite, connected, and acyclic digraph
in which every vertex has constant in- and out-degree~$d\ge2$.
In other words, $\RT>$ is the Cayley
graph of the free group
$F_d:=\langle\gen_1,\ldots,\gen_d\rangle$
where unlike its symmetric version,
the generators have no relations (see~\cref{fig:d2regtree}).
We identify the vertex set of~$\RT>$ with~$F_d$,
the root vertex $\rv\in\RT>$ corresponding to the identity element,
and we let~$\RTA>$ denote the adjacency matrix.
\begin{theorem}[Convergence of star moments
  for uniform $d$-regular digraphs]\label{thm:moment-method}
  For every $k\ge0$ and every $w\in\Sigma^k$,
  \[
    \frac1n\EE{\Tr{\RGA>^w}}
    \conv\MM(w):=\RTA>^w(\rv,\rv),
  \]
  where $\RTA>^w(v,v')$ is defined
  for any pair of vertices~$v,v'\in\RT>$ by
  \[
  \RTA>^w(v,v')\qquad:=%
  \sum_{\substack{v_0,v_1\ldots,v_{k-1},v_k\in\RT>\\%
  v_0=v,\ v_k=v'}}
    \RTA>^{w_1}(v_0,v_1)\RTA>^{w_2}(v_1,v_2)%
    \cdots%
    \RTA>^{w_k}(v_{k-1},v_k).
  \]
\end{theorem}

\noindent%
Note that $\RTA>^w(v,v')=\II_{\{v=v'\}}$
if $w=\ew\in\Sigma^0$
is the empty word.
In any case,
all summands of~$\RTA>(v,v')$ are either~$0$ or~$1$,
and each non-zero summand corresponds to a
solution $(i_1,\ldots,i_k)\in{[d]}^k$ to the
word problem
\[v\cdot\gen_{i_1}^{w_1}\cdots\gen_{i_k}^{w_k}=v'\]
in the free group~$F_d$, where~$\gen_i^{*\vphantom1}$
denotes the inverse of~$\gen_i^{\vphantom1}=:\gen_i^1$.
We call such a solution $(i_1,\ldots,i_k)$ a $w$-path
from~$v$ to~$v'$,
which we can also picture as
\begin{align*}
  v&=:v_0\stackrel{w_1}\Dash v_1\stackrel{w_2}\Dash\cdots
  \stackrel{w_{k-1}}\Dash v_{k-1}\stackrel{w_k}\Dash v_k:=v',
\end{align*}
where $v_j:=v\cdot\gen_{i_1}^{w_1}\cdots\gen_{i_j}^{w_j}$
for all $0\le j\le k$.
In plain words, $\MM(w)$ is the cardinal of the set~$\GP(w)$ of all
$w$-paths from~$\rv$ to~$\rv$ (or from any other vertex to itself,
by transitivity of the Cayley graph~$\RT>$).
We stress that we do not consider any randomness on~$\RT>$:
in this respect, our purpose is different from
Kesten~\cite{kesten1959symmetric},
who studied spectral properties of random walks
on the undirected regular tree~$\RT$.

In fact, \cref{thm:moment-method}
is a consequence of the following
general criterion,
similar to~\cite[Theorem~1.1]{mckay1981expected}:
under a growth assumption on the number of
short cycles, we show that the star moments of
deterministic $d$-regular digraphs
converge
to the star moments of the
$d$-regular directed tree~$\RT>$.
\begin{theorem}[Convergence of star moments
  for deterministic $d$-regular digraphs]\label{thm:criterion}
  Let~$G_n,\,n\ge1,$ be a $d$-regular digraph
  with adjacency matrix~$A_n$
  on a vertex set~$V_n$%,
  %with $|V_n|\to\infty$ as $n\to\infty$
  .
  Let $k\ge1$ and
  suppose that for every $j\in[k]$, the number~$c_j(G_n)$
  of cycles with length~$j$ in~$G_n$
  (see~\eqref{eq:nbcycles})
  fulfills
  \begin{align}
    \frac{c_j(G_n)}{|V_n|}&\conv0.\label{eq:fewcycles}
  \intertext{Then for every word~$w\in\Sigma^k$,}
    \frac1{|V_n|}\Tr{A_n^w}&\conv\MM(w).
  \end{align}
\end{theorem}

Our last result is a combinatorial derivation of
a formula for~$\MM(w)$, which requires
some notation.
Recall that a partition~$\pi$ of~$[k]$
can also be seen as the equivalence relation~$\sim_\pi$
on~$[k]$ such that
$i\sim_\pi j\iff\exists V\in\pi,\ \{i,j\}\subseteq V$
for all $i,j\in[k]$.
We say that~$\pi$ is \emph{non-crossing},
written~$\pi\in\NC(k)$,
if
$i_1\sim_\pi j_1,\ i_2\sim_\pi j_2
\implies j_1\sim_\pi i_2$
for all $1\le i_1<i_2<j_1<j_2\le k$.
The cardinal~$|\NC(k)|$ is equal to the ubiquitous
Catalan number $C_k:=\frac1{k+1}\binom{2k}k$,
which is also~\cite{stanley2015catalan} the cardinal~$|\NC_2(2k)|$
of the non-crossing \emph{pair} partitions
of~$[2k]$ (where each block has size~$2$).
We further say that~$\pi$
is an \emph{alternating non-crossing partition} of~$w$
($\pi\in\ANC(w)$) if for
every block $V:=\{i_1<\ldots<i_m\}\in\pi$,
the subword~$w|_V:=w_{i_1}\cdots w_{i_m}$ of~$w$
is \emph{alternating},
that is either of the form $w|_V=1{*}\cdots1{*}$ or $w|_V={*}1\cdots{*}1$
(so~$w$ and all blocks of~$\pi$ must have even size).%
\begin{theorem}[Combinatorial formula for~$\MM(w)$]\label{thm:starmoments}
  For every $k\ge0$ and every $w\in\Sigma^k$,
  \[
    \MM(w)\:\:\:\:=\sum_{\pi\in\ANC(w)}
      \left(\prod_{V\in\pi}{(-1)}^{\frac{|V|}2-1}\,%
      C_{\frac{|V|}2-1}\right)d^{|\pi|}.
  \]
\end{theorem}
\noindent
Notably, our proof shows that the $w$-paths may be counted
by an inclusion-exclusion principle involving
non-crossing pair partitions, thus explaining the
presence of signs and Catalan numbers.

We mention that \cref{thm:moment-method,thm:starmoments}
may be recovered
by taking a detour to free probability
from a theorem of Nica~\cite{nica1993asymptotically},
see \cref{sec:alternativeproof},
where we also put the
oriented Kesten--McKay conjecture
in more context.
Our main motivation for this work is to provide
direct combinatorial proofs, which we do in \cref{sec:directproof}.

\paragraph*{\bfseries Acknowledgments.}
We would like to thank Pierre Youssef
for suggesting to work on this problem and Charles Bordenave
for informing us about~\cite{nica1993asymptotically}.

\section{Direct combinatorial proofs}\label{sec:directproof}
\subsection{Convergence of star moments.}
In this section, we prove \cref{thm:criterion}
and its corollary, \cref{thm:moment-method}.
Let~$G$ be a multigraph with adjacency matrix~$A$
and vertex set~$V$. We call a sequence of
$j\ge1$ distinct arcs
$\eps_1,\ldots,\eps_j$
(read in any cyclic order)
a \emph{plain cycle of length~$j$ in~$G$}
if each of the pairs
$\{\eps_1,\eps_2\},\ldots,\{\eps_{j-1},\eps_j\},\{\eps_j,\eps_1\}$
has a common vertex
(we disregard the arc orientations).
Discounting the cyclic orderings,
the number of plain cycles with length~$j$
in~$G$ is then given by
\begin{equation}
  c_j(G):=\frac1{2j}\sum_{w\in\Sigma^j}\sum_{\mathbf v}
    A^{w_1}(v_0,v_1)\cdots A^{w_j}(v_{j-1},v_j),
  \label{eq:nbcycles}
\end{equation}
the second summation ranging over every
$\mathbf v:=(v_0,\ldots,v_{j-1},v_j=v_0)\in V^{j+1}$
such that the sequence
${\bigl({(v_{i-1},v_i)}^{w_i}\bigr)}_{1\le i\le j}$
is injective, where ${(v,v')}^1:=(v,v')$ and ${(v,v')}^*:=(v',v)$
for all $v,v'\in V$.
\begin{proof}[Proof of \cref*{thm:criterion}]
  Write $v\in\cC_{n,k}$ if there exists a
  vertex~$v'\in V_n$ at distance at most~$k$
  from~$v$ (i.e., $A_n^{w'}(v,v')>0$ for some $w'\in\Sigma^j$,
  with $j\le k$)
  and belonging to a cycle of length
  at most~$k$ in~$G_n$.
  Note that, by union bound,
  \[
    |\cC_{n,k}|\le\sum_{i=1}^k{(2d)}^i\sum_{j=1}^kj\,c_j(G_n)
    \le k^2{(2d)}^k\sum_{j=1}^kc_j(G_n).
  \]
  Thus, on the one hand,
  \begin{align}
    \frac1{|V_n|}\sum_{v\in\cC_{n,k}}A^w_n(v,v)
    \le k^2{(2d^2)}^k\,\frac{\sum_{j=1}^kc_j(G_n)}{|V_n|}
    &\conv0,\label{eq:trbadpart}
  \intertext{using the trivial upper bound~$A_n^w(v,v)\le d^k$
  and \eqref{eq:fewcycles}.
  On the other hand, for $v\notin\cC_{n,k}$, no
  vertex accessible in at most~$k$ steps from~$v$
  belongs to a cycle of length at most~$k$. Since~$G_n$
  is $d$-regular, the ball $B_{G_n}(v,k)$ of radius~$k$ around~$v$
  must then look exactly like the ball~$B_{\RT>}(\rv,k)$.
  In particular, $A_n^w(v,v)=\RTA>^w(\rv,\rv)$ for all $v\in V_n\setminus\cC_{n,k}$,
  and thus}
    \frac1{|V_n|}\sum_{v\not\in\cC_{n,k}}A^w_n(v,v)
    =\frac{|V_n|-|\cC_{n,k}|}{|V_n|}\,\RTA>^w(\rv,\rv)
    &\conv
    \MM(w).\label{eq:trgoodpart}
  \end{align}
  Adding~\cref{eq:trbadpart,eq:trgoodpart} then shows
  as stated that
  \[
    \frac1{|V_n|}\Tr{A_n^w}
    =\frac1{|V_n|}\sum_{v\in\cC_{n,k}}A^w_n(v,v)+
    \frac1{|V_n|}\sum_{v\notin\cC_{n,k}}A^w_n(v,v)
    \conv
    \MM(w).\qedhere
  \]
\end{proof}
\begin{remark}
  As we can see from the proof, the condition~\eqref{eq:fewcycles}
  implies more generally that
  $G_n\to\RT>$ with respect to the ``oriented''
  Benjamini--Schramm topology: for every $k\ge1$,
  the balls of radius~$k$ in~$G_n$ are eventually isomorphic
  to the ball of radius~$k$ in~$\RT>$. As such, \cref{thm:criterion}
  constitutes the non-symmetric version
  of~\cite[Proposition~14]{abert2016measurable}.
\end{remark}

Next, we show that the growth condition~\eqref{eq:fewcycles}
of~\cref{thm:criterion} holds
in expectation for the uniform $d$-regular digraph~$\RG>$.
\begin{lemma}[$\RG>$ has few short cycles on average]\label{lem:fewcycles}
  For every $k\ge1$,
  \[
    \frac1n\EE{c_k\left(\RG>\right)}\conv0.
  \]
\end{lemma}
\begin{proof}
  To estimate this expectation, it is convenient to work with the
  so called \emph{configuration model}~$\CM$, whose construction
  we briefly recall. First, to each vertex $i\in[n]$ we attach~$d$
  unique incoming half-arcs $\eps^+_{i+(p-1)n},\,p\in[d],$
  and another~$d$ unique outgoing half-arcs
  $\eps^-_{i+(q-1)n},\,q\in[d]$.
  Second, we choose uniformly at random
  a bijection~$f_{d,n}$ joining each
  of the~$nd$ outgoing half-arcs to one of the~$nd$
  incoming half-arcs
  (so there are~${(nd)}!$
  possible choices for the bijection~$f_{d,n}$).
  This gives rise to a random
  multigraph~$\CM$ on~$[n]$
  in which the number of arcs $\CMA(i,j)$
  from~$i$ to~$j$
  equals the number of pairs $(p,q)\in[d]^2$ such that
  $f\bigl(\eps^+_{i+(p-1)n}\bigr)=\eps^-_{j+(q-1)n}$.
  Also, the distribution of~$\CM$ conditional on
  the event
  \[
    \Simple:\text{``$\CM$ is simple''}
    =\Bigl\{\text{$\CMA(i,i)=0$ and $\CMA(i,j)\le1$ for all $i\neq j\in[n]$}\Bigr\}
  \]
  coincides with the law of~$\RG>$:
  $\cL(\RG>)
    =\cL(\CM\mid\Simple)$.
  Now,
  the expected number~$\EE{c_k\bigl(\CM\bigr)%
  }$
  of cycles with length~$k$
  is easy to estimate from~\eqref{eq:nbcycles}%
  :
  \begin{align*}
    \EE{c_k\left(\CM\right)
    }
    &\le\frac1{2k}\cdot2^k\cdot n^k\cdot
    d^k\PP\left(f_{d,n}(\eps^+_1)=\eps^-_2,
    \ldots,f_{d,n}(\eps^+_{k-1})=\eps^-_k,f_{d,n}(\eps^+_k)=\eps^-_1\right)\\
    &=\frac{{(2nd)}^k\,{(nd-k)}!}{2k\,{(nd)}!}\\
    &\sim\frac{2^{k-1}}k,
  \end{align*}
  by Stirling's formula.
  Furthermore, the probability~$\PP(\Simple)$
  of~$\CM$ being simple was computed in~\cite{janson2009probability}
  and is known~\cite{cooper2004size} to be bounded away from zero
  as $n\to\infty$.
  Hence
  \[\frac1n\EE{c_k\left(\RG>\right)}
  =\frac1n\EE\left[{c_k\left(\CM\right)}\;\middle|\;\Simple\right]
  \le\frac{\EE{c_k\left(\CM\right)%
  }}{n\PP\left(\Simple\right)}
  \,\xrightarrow[n\to\infty]{}\;0,\]
  which concludes the proof.
\end{proof}
We are now ready to prove \cref{thm:moment-method}.
\begin{proof}[Proof of \cref*{thm:moment-method}]
  Let~${(n_p)}_{p\ge1}$ be an increasing sequence of integers
  tending to $\infty$. By \cref{lem:fewcycles}, the
  convergence
  \begin{align*}
    \frac1{n_p}\,c_k\left(\RG>[d,n_p]\right)
    &\conv[p]0
  \intertext{holds in expectation, and thus
  also in probability. A classical application
  of the Borel--Cantelli lemma shows that
  it further holds almost surely along a subsequence:
  there exists ${(n'_p)}_{p\ge1}\subseteq{(n_p)}_{p\ge1}$
  such that}
    \frac1{n'_p}\,c_k\left(\RG>[d,n'_p]\right)
    &\conv[p]0
  \intertext{almost surely. Then \cref{thm:criterion} entails that}
    \frac1{n'_p}\Tr{\RGA>[d,n'_p]^w}
    &\conv[p]
    \MM(w)
  \intertext{holds almost surely.
  Since $\Tr{\RGA>[d,n'_p]^w}\le n'_p\,d^k$,
  the dominated convergence theorem then yields}
  \frac1{n'_p}\EE{\Tr{\RGA>[d,n'_p]^w}}
    &\conv[p]
    \MM(w).
\end{align*}
We have just showed that every subsequence
of $\frac1n\EE{\Tr{\RGA>^w}},\,n\ge1,$ admits a further subsequence
converging to~$\MM(w)$, so \cref{thm:moment-method} is proved.
\end{proof}

\subsection{Combinatorial formula for~$\MM(w)$.}
Before establishing \cref{thm:starmoments},
let us warm up with a simple necessary condition
for the set~$\GP(w)$ of $w$-paths from~$\rv$ to~$\rv$
to be non-empty.
\begin{lemma}\label{lem:balanced}
  If~$w:=w_1\cdots w_k$ is a word on~$\Sigma$ such that $\GP(w)\neq\emptyset$,
  then~$w$ is balanced: $w\in{\{1{*},{*}1\}}^p$ with~$k=2p$.
\end{lemma}
\begin{proof}
  We proceed by induction on~$k$.
  The lemma holds trivially if $k=0$.
  Suppose $k\neq 0$. By assumption, there exists a $w$-path
  $\mathbf p:=(i_1,\ldots,i_k)$
  from~$v_0:=\rv$ to itself:
  \[\mathbf p\enspace=\enspace v_0\stackrel{w_1}\Dash v_1
  \stackrel{w_2}\Dash\cdots\stackrel{w_k}\Dash v_k=v_0,\]
  with $v_j:=\gen_{i_1}^{w_1}\cdots\gen_{i_j}^{w_j},\,
  0\le j\le k$, where we recall that
  the generators~$\gen_1,\ldots,\gen_d$
  have no relations.
  In particular $v_1\neq v_0$ (because $e_{i_1}\neq\rv$), and
  thus $k\ge2$. Let $2\le r\le k$ be the smallest index
  for which~$v_r=v_0$, so
  \[\rv=v_0=v_r:=\gen_{i_1}^{w_1}\left(\gen_{i_2}^{w_2}\cdots\gen_{i_{r-1}}^{w_{r-1}}\right)\gen_{i_r}^{w_r},
  \qquad\text{that is,}\quad
  \gen_{i_r}^{w_r}\gen_{i_1}^{w_1}={\left(\gen_{i_2}^{w_2}\cdots\gen_{i_{r-1}}^{w_{r-1}}\right)}^{-1}.\]
  Since the generators have no relations,
  this forces
  $i_1=i_r$ and $w_1\neq w_r$ (in other words,
  $\RT>$ has no cycle so
  the arc taken in
  $v_0\stackrel{w_1}\Dash v_1$ must match the one
  in $v_{r-1}\stackrel{w_r}\Dash v_r=v_0$).
  Thus~$w$ is of the form
  $w=1u{*}v$ or $w={*}u1v$ where $u=w_2\cdots w_{r-1}$
  and $v:=w_{r+1}\cdots v_k$, and
  \begin{equation}
    \mathbf p\enspace=\enspace v_0\stackrel{w_1}\Dash\underbrace{v_1%
    \stackrel{w_2}\Dash\cdots\stackrel{w_{r-1}}\Dash v_{r-1}}_{\text{$u$-path}}
    \stackrel{w_r}\Dash v_r=
    \underbrace{v_0\stackrel{w_{r+1}}\Dash v_{r+1}%
    \stackrel{w_{r+2}}\Dash\cdots\stackrel{w_k}\Dash v_k}_{\text{$v$-path}}
    =v_0.\label{eq:decomposition}
  \end{equation}
  By induction,
  the smaller words~$u$ and~$v$ are balanced, and
  thus~$w$ is also balanced.
\end{proof}

A consequence of \cref{lem:balanced} is
that $\MM(w)=0$ if~$w$ is not balanced,
so \cref{thm:starmoments}
is proved for such a word since then~$\ANC(w)=\emptyset$.
Note also that \cref{thm:starmoments} holds
if~$w$ is the empty word~$\ew$, because $\ANC(\ew):=\{\emptyset\}$
is reduced to the empty partition and~$\GP(\ew):=\{\ep\}$
is reduced to the empty path.
We henceforth assume~$w:=w_1\cdots w_{2p}$ non-empty
and balanced.
The decomposition~\eqref{eq:decomposition} of
a $w$-path (from~$\rv$ to~$\rv$)
with respect to its first return time
to the origin is clearly unambiguous.
Putting aside the choice of vertices along the path,
this gives rise to a ``skeleton''
which as we now claim can be encoded as a
certain partition~$\pi\in\ANC(w)$
whose every block~$V\in\pi$ has cardinal $|V|=2$;
we write $\pi\in\ANC_2(w)$
and call it an \emph{alternating non-crossing pair partition}
of~$w$.
Specifically, 
let $\mathbf p:=(i_1,\ldots,i_{2p})\in\GP(w)$
and denote by $r\in\{2,\ldots,2p\}$
its first return time to~$\rv$, that is
$v_j:=\gen_{i_1}^{w_1}\cdots\gen_{i_j}^{w_j}\neq\rv$
for every $1\le j<r$, and $v_r=\rv$. Then
the skeleton of~$\mathbf p$ is defined inductively
as
\[
    \skel(\mathbf p):=\Bigl\{\{1,r\}\Bigr\}
      \;\scalebox{1.25}{$\cup$}\;
      \Bigl\{V+1:V\in\skel[u](i_2,\ldots,i_{r-1})\Bigr\}
      \;\scalebox{1.25}{$\cup$}\;
      \Bigl\{V+r:V\in\skel[v](i_{r+1},\ldots,i_{2p})\Bigr\},
\]
where $u:=w_2\cdots w_{r-1}$
and $v:=w_{r+1}\cdots w_{2p}$,
with the base case $\skel[\ew](\ep):=\emptyset$ for the unique $\ew$-path~$\ep$.
Essentially, a block $V:=\{j<k\}$ in~$\skel(\mathbf p)$
means that $v_j,\ldots,v_{k-1}\neq v_{j-1}=v_k$,
i.e.,~$k$ is the first return time to the vertex visited
at time~$j-1$.

Conversely, given
an alternating non-crossing \emph{pair}
partition~$\pi\in\ANC_2(w)$
of~$w$,
what is $\skel^{-1}\{\pi\}$,
the subset of $w$-paths~$\mathbf p:=(i_1,\ldots,i_{2p})\in\GP(w)$
with skeleton~$\skel(\mathbf p)=\pi$?
Clearly, since each block $V:=\{j<k\}\in\pi$ indicates
a segment of the path where it exits and first returns
to the vertex~$v_{j-1}$, we must have $i_j=i_k$,
i.e.,
the arc taken at time~$j$ to exit $v_{j-1}$
must be taken again at time~$k$ (but ``backwards'',
since $w_j\neq w_k$) in order to return to~$v_{j-1}$.
This condition alone does of course not
prevent a premature return
$v_{j'}=v_{j-1}$ for some $j'\in\{j+1,\ldots,k-1\}$.
A premature return at time~$j'$ can however only happen
if
\begin{enumerate}[label=(\roman*),wide]
    \item $w_{j'}\neq w_j$ (the arc at time~$j'$ must be taken in the opposite direction as when exiting~$v_{j-1}$), and
    \item $j'$ is the lower element in its block, $U:=\{j'<k'\}\in\pi$,
    which is directly surrounded by~$V$: $j<j'<k'<k$ and there is no other block $\{j''<k''\}\in\pi$
with $j<j''<j'<k'<k''<k$.
\end{enumerate}
In case~(i) and~(ii) hold,
we write $j\bad_\pi j'$ as well as $U\in B(\pi)$,
and say that~$j,j'$ form a
\emph{bad pair} and that~$U$ is a \emph{bad block}.
See \cref{fig:pathskel} for an illustration.
% Figures fig:path, fig:skel
\begin{figure}\centering
\begin{subfigure}[b]{.4\textwidth}\centering%
\includegraphics{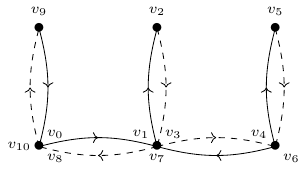}
\subcaption{$\mathbf p=(i,j,j,k,\ell,\ell,k,i,m,m),\,k\notin\{i,\ell\}.$\\
(Arcs in~$\RT>^*$
are drawn with a dashed line.)}\label{fig:path}
\end{subfigure}\hfill%
\begin{subfigure}[b]{.6\textwidth}\centering%
\includegraphics{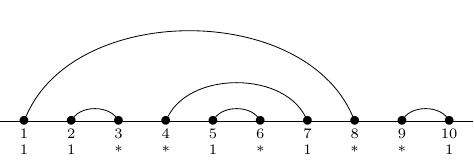}
\subcaption{$\pi=\bigl\{\{1,8\},\{2,3\},\{4,7\},\{5,6\},\{9,10\}\bigr\}$,\\with $\bad_\pi=\{(1,4),(4,5)\}$
and $B(\pi)=\bigl\{\{4,7\},\{5,6\}\bigr\}$.}\label{fig:skel}
\end{subfigure}
\caption{(\textsc{a}) A $w$-path~$\mathbf p$ and (\textsc{b}) its skeleton~$\pi:=\skel(\mathbf p)$,
for~$w:=11{*}{*}1{*}1{*}{*}1$.}\label{fig:pathskel}
\end{figure}

Summarizing, for a $w$-path $(i_1,\ldots,i_{2p})$
to have skeleton~$\pi$,
we must have $i_j\neq i_{j'}$ if~$j,j'$ form a bad pair (i.e., $j\bad_\pi j'$),
and $i_j=i_{j'}$ if~$j,j'$ belong to the same block ($j\sim_\pi j'$).
It should be clear that these requirements are also sufficient:
\begin{lemma}\label{lem:decomposition}
  The map~$\skel:\GP(w)\to\ANC_2(w)$ is surjective: for every $\pi\in\ANC_2(w)$,
    \begin{align}
      \skel^{-1}\{\pi\}
      &=\left\{(i_1,\ldots,i_{2p})\in{[d]}^{2p}\enspace\middle|\enspace
      \forall(j,j')\in{[2p]}^2,\:
      \biggl\{
        \begin{aligned}
          j\sim_\pi{}&j'\implies i_j=i_{j'}\\[-.5em]
          j\bad_\pi{}&j'\implies i_j\neq i_{j'}
        \end{aligned}\right\}.\label{eq:preimage}
      \intertext{Furthermore,}
      \left|\skel^{-1}\{\pi\}\right|
      &=\prod_{V\in\pi}\left(d-\II_{\{V\in B(\pi)\}}\right).\label{eq:cardinal}
    \end{align}
\end{lemma}
\begin{proof}
  First, the expression given for the cardinal~\eqref{eq:cardinal}
  is always positive because $d\ge2$,
  and is easily derived from~\eqref{eq:preimage}:
  for each block~$V:=\{j<k\}\in\pi$,
  there are~$d$ degrees
  of freedom for the choice of $i_j=i_k\in[d]$,
  except if~$V$ is a bad block, in which case there
  is one degree of freedom less (because $i_j=i_k$ must
  be different from~$i_{j'}$, where $j\bad_\pi j'$).
  It remains to prove~\eqref{eq:preimage},
  which we do by induction on the balanced word
  $w:=w_1\cdots w_{2p}$.
  There is nothing to prove if~$p=0$. Suppose $p\ge1$ and
  consider the decomposition of~$\pi$ with respect to the block containing~$1$,
  \[
    \pi:=\bigl\{\{1,r\}\bigr\}\cup\bigl\{V+1:V\in\pi^{(u)}\bigr\}\cup\bigl\{V+r:V\in\pi^{(v)}\bigl\},
  \]
  where $u:=w_2\cdots w_{r-1}$, $v:=w_{r+1}\cdots w_{2p}$, and
  \begin{alignat*}{3}
    &\pi^{(u)}&&:=\Bigl\{V-1:V\in\pi,\ V\subseteq\{2,\ldots,r-1\}\Bigr\}&&\in\ANC_2(u),\\[.4em]
    &\pi^{(v)}&&:=\Bigl\{V-r:V\in\pi,\ V\subseteq\{r+1,\ldots,2p\}\Bigr\}&&\in\ANC_2(v).
  \end{alignat*}
  For $j\in[r-2]$, write $j\in J$
  if $w_{j+1}\neq w_1$
  and the block~$\{j<k\}\in\pi^{(u)}$ containing~$j$
  in~$\pi^{(u)}$ is not surrounded by any other block
  (i.e., there is no $\{j'<k'\}\in\pi^{(u)}$
  with $j'<j<k<k'$).
  Because of the previous decomposition
  and the definition of~$\bad_\pi$, we then have
  \[\bad_\pi=\bigl\{(1,j+1):j\in J\bigr\}\cup\bigl\{(j+1,j'+1):j\bad_{\pi^{(u)}}j'\bigr\}
  \cup\bigl\{(j+r,j'+r):j\bad_{\pi^{(v)}}j'\bigr\}.\]
  Recall also that~$\{1,r\}\in\pi$ indicates that the
  $w$-paths
  with skeleton~$\pi$ first return to~$\rv$ at time~$r$.
  Thus
  \begin{align*}
    (i_1,\ldots,i_{2p})\in\skel^{-1}\{\pi\}
    &\iff
      \left\{
      \begin{aligned}
        &i_1=i_r,\ \forall j\in J,\ i_{j+1}\neq i_1,\\[-.5em]
        &(i_2,\ldots,i_{r-1})\in\skel[u]^{-1}\bigl(\pi^{(u)}\bigr),\ 
        (i_{r+1},\ldots,i_{2p})\in\skel[v]^{-1}\bigl(\pi^{(v)}\bigr),
      \end{aligned}\right.
    \intertext{and, by the induction hypothesis,}
    (i_1,\ldots,i_{2p})\in\skel^{-1}\{\pi\}
    &\iff\left\{
    \begin{aligned}    
    &i_1=i_r,\ \forall j\in J,\ i_{j+1}\neq i_1,\\[-.5em]
    &\forall(j,j')\in{\{2,\ldots,r-1\}}^2,\ 
    \biggl\{
    \begin{aligned}
        (j-1)\sim_{\pi^{(u)}}{}&(j'-1)\implies i_j=i_{j'},\\[-.5em]
        (j-1)\bad_{\pi^{(u)}}{}&(j'-1)\implies i_j\neq i_{j'},
    \end{aligned}\\
  &\forall(j,j')\in{\{r+1,\ldots,2p\}}^2,\ 
    \biggl\{
    \begin{aligned}
        (j-r)\sim_{\pi^{(v)}}{}&(j'-r)\implies i_j=i_{j'},\\[-.5em]
        (j-r)\bad_{\pi^{(v)}}{}&(j'-r)\implies i_j\neq i_{j'},
    \end{aligned}
  \end{aligned}\right.\\[.4em]
  &\iff\forall(j,j')\in{[2p]}^2,\:
  \biggl\{
    \begin{aligned}
      j\sim_\pi{}&j'\implies i_j=i_{j'},\\[-.5em]
      j\bad_\pi{}&j'\implies i_j\neq i_{j'}.
    \end{aligned}\qedhere%\\[.4em]
 \end{align*}
\end{proof}
We can now complete the proof of \cref{thm:starmoments}.
\begin{proof}[Proof of \cref*{thm:starmoments}]
Let $w:=w_1\cdots w_{2p}$ be a balanced word on~$\Sigma$.
It follows from \cref{lem:decomposition} that the set~$\GP(w)$ of $w$-paths
from~$\rv$ to~$\rv$ may be partitioned with respect to their skeleton as
\begin{align*}
  \GP(w)\quad&=\bigsqcup_{\pi\in\ANC_2(w)}\skel^{-1}\{\pi\},
\intertext{and passing to the cardinal, we get}
  \MM(w)\quad
  &=\sum_{\pi\in\ANC_2(w)}\prod_{V\in\pi}\left(d-\II_{\{V\in B(\pi)\}}\right)\quad
  =\sum_{\pi\in\ANC_2(w)}
\sum_{A\subseteq B(\pi)}{(-1)}^{|A|}\,d^{|\pi|-|A|},
\intertext{%
by expanding out the product%
\footnotemark.
Now, given~$\pi\in\ANC_2(w)$ and~$A\subseteq B(\pi)$,
we construct a coarser partition~$\pi':=\gamma(\pi,A)$
from~$\pi$ by merging, for each bad pair~$(j,j')\in A$,
the block containing~$j'$ into its surrounding
block (the one containing~$j$). In other words,
$\sim_{\pi'}$ is the smallest equivalence relation on~$[2p]$
containing ${\sim_{\pi}}\cup A$.
For instance, if~$\pi$ is the partition of~\cref{fig:skel}
and~$A:=\{(4,5)\}$,
then $\pi':=\gamma(\pi,A)=\bigl\{\{1,8\},\{2,3\},\{4,5,6,7\},\{9,10\}\bigr\}$.
It is clear that the conditions~(i) and~(ii) of forming
a bad pair guarantee that~$\pi'$ remains
non-crossing
and
alternating w.r.t.~$w$:
$\pi'\in\ANC(w)$. Further, $\pi'$ has
exactly~$|A|$ fewer blocks than~$\pi$,
which has~$p$ blocks,
so ${(-1)}^{|A|}\,d^{|\pi|-|A|}={(-1)}^{p-|\pi'|}\,d^{|\pi'|}$.
Conversely, given~$\pi'\in\ANC(w)$,
any pair partition~$\pi$ which is \emph{finer} than~$\pi'$
(i.e., ${\sim_\pi}\subseteq{\sim_{\pi'}}$)
automatically leads to
an alternating non-crossing pair partition
$\pi\in\ANC_2(w)$ of~$w$
having a certain set of bad pairs.
Therefore,}%
\noalign{\footnotetext{At the level of sets, this amounts to writing
\[
  \skel^{-1}\{\pi\}=
  \Bigl\{(i_1,\ldots,i_{2p})\in{[d]}^{2p}:
  \forall(j,j')\in{[d]}^{2p},\ j\sim_\pi j'\implies i_{j}=i_{j'}
  \Bigr\}\;
  \scalebox{1.5}{$\setminus$}\bigcup_{j\bad_\pi j'}%
  \Bigl\{(i_1,\ldots,i_{2p})\in{[d]}^{2p}:
  i_j=i_j'
  \Bigr\}
\]
and using the inclusion-exclusion formula.}}
  \MM(w)\quad
  &=\sum_{\pi'\in\ANC(w)}{(-1)}^{p-|\pi'|}\,d^{|\pi'|}
  \sum_{\substack{\pi\in\ANC_2(w)\\\pi\finer\pi'}}\:
  \sum_{\substack{A\subseteq B(\pi)\\\gamma(\pi,A)=\pi'}}1,
\end{align*}
where we wrote $\pi\finer\pi'$
for ${\sim_\pi}\subseteq{\sim_{\pi'}}$.
Since
\[{(-1)}^{p-|\pi'|}\:
=\prod_{V\in\pi'}{(-1)}^{\frac{|V|}2-1},\]
it remains to observe that
\begin{equation}
  \sum_{\substack{\pi\in\ANC_2(w)\\\pi\finer\pi'}}\:
  \sum_{\substack{A\subseteq B(\pi)\\\gamma(\pi,A)=\pi'}}1
\quad=\quad\prod_{V\in\pi'}C_{\frac{|V|}2-1}
\label{eq:product-catalan}
\end{equation}
to conclude. But
constructing~$\pi\in\ANC_2(w)$ such that
$\pi\finer\pi'$
and $\gamma(\pi,A)=\pi'$ for some~$A\subseteq B(\pi)$
is equivalent
to partitioning each block~$V:=\{i_1,\ldots,i_{2m}\}\in\pi'$
using an alternating pair partition of~$w|_V$
containing the block~$\{1,2m\}$.
Since~$w|_V$ is already alternating (because $\pi'\in\ANC(w)$),
this amounts to choosing a non-crossing pair partition
of~$\{2,\ldots,2m-1\}$,
i.e., an element of~$\NC_2(2m-2)$.
Then~\eqref{eq:product-catalan}
follows from
the well-known fact $|\NC_2(2m-2)|=C_{m-1}$,
see~\cite[Exercise~61]{stanley2015catalan}.
\end{proof}

\section{Free probability and the oriented Kesten--McKay conjecture}\label{sec:alternativeproof}
\subsection{Free probability.}
Let us start this concluding section by showing
how \cref{thm:moment-method,thm:starmoments}
may be recovered from Nica's work~\cite{nica1993asymptotically}.
Free probability is a vast field initiated by Voiculescu;
we only introduce the bare minimum and
refer to~\cite{nica2006lectures}
for detail.
The general framework is that
of non-commutative variables $x,y,\ldots$ in
some unital algebra~$\fA$ endowed with an adjoint operator~$^*$
and a linear form $\varphi:\fA\to\CC$
such that $\varphi(1)=1$ and $\varphi(x^*x)\ge0$
for all $x\in\fA$. The pair $(\cA,\varphi)$
is called a \emph{non-commutative probability space},
where the \emph{state}~$\varphi$ plays the
rôle of an expectation. The \emph{distribution}
of~$x$ or more generally, the (joint) distribution
of $(x_1,\ldots,x_k)$ is given by all mixed moments
$\varphi(x_{i_1}\cdots x_{i_\ell})$ for $\ell\ge1$
and $(i_1,\ldots,i_\ell)\in{[k]}^\ell$,
which themselves may be expressed
through the moment-cumulant formula
\cite[Lecture~11]{nica2006lectures}:
\begin{equation}
  \varphi(x_1\cdots x_k)\enspace=
  \sum_{\pi\in\NC(k)}\prod_{\substack{V\in\pi\\V:=\{i_1<\cdots<i_\ell\}}}\kappa_\ell(x_{i_1},\ldots,x_{i_\ell}),
\label{eq:moment-cumulant}
\end{equation}
where the \emph{free cumulants}~$\kappa_\ell:\fA^\ell\to\CC,\,\ell\ge1,$
are defined inductively so that~\eqref{eq:moment-cumulant}
holds for any $k\ge1$
and any non-commutative variables $x_1,\ldots,x_k\in\fA$.
Similarly to the log-Laplace transform
of classical random variables,
the free cumulants of $(x_1,\ldots,x_k)$
may be gathered into the so called \emph{$R$-transform}
\cite[Lecture~16]{nica2006lectures}:
\begin{equation}
R_{(x_1,\ldots,x_k)}(z_1,\ldots,z_k)
:=\sum_{\ell=1}^\infty
\sum_{i_1,\ldots,i_\ell\in[k]}\kappa_\ell(x_{i_1},\ldots,x_{i_\ell})\,
z_{i_1}\cdots z_{i_\ell},\label{eq:r-transform}
\end{equation}
which is a formal series in non-commutative indeterminates $z_1,\ldots,z_k$.
Analogously to independence for classical random variables,
$x_1,\ldots,x_k$ are \emph{free}
if
$R_{(x_1,\ldots,x_k)}(z_1,\ldots,z_k)=R_{x_1}(z_1)+\cdots+R_{x_k}(z_k)$,
which is commonly phrased by the sentence ``mixed cumulants vanish''
(i.e., $\kappa_\ell(x_{i_1},\ldots,x_{i_\ell})=0$
for every $\ell\ge1$
and every non-constant sequence $(i_1,\ldots,i_\ell)\in{[k]}^\ell$).

Nica~\cite{nica1993asymptotically}
showed that for
$P_{1,n},\ldots,P_{d,n}\in\{0,1\}^{n\times n}$
uniform, independently chosen
permutation matrices,
there exists a non-commutative
probability space~$(\fA,\varphi)$
and free variables $u_1,\ldots,u_d\in\fA$
such that:
\begin{enumerate}[label=(\alph*), wide]
\item There is the convergence of
mixed moments
\[
  \frac1n
  \EE{\Tr{P_{i_1,n}^{w_1}\cdots P_{i_k,n}^{w_k}}}
  \conv
  \varphi\left(u_{i_1}^{w_1}\cdots u_{i_k}^{w_k}\right),
\]
for every $k\ge1$, every $(i_1,\ldots,i_k)\in{[d]}^k$,
and every word $w\in\Sigma^k$.
\item The~$u_i$'s are \emph{Haar unitaries},
in the sense that $u_i^*u_i^{\vphantom*}=u_i^{\vphantom*}u_i^*=1$
and $\varphi(u_i^k)=0$ for every $k\ge1$.
\end{enumerate}
It follows from~(a) and linearity that
the star moments of
$\CMA:=P_{1,n}+\cdots+P_{d,n}$
converge to those of
$a_d:=u_1+\cdots+u_d$, and it is easy to see that~$\CMA$
is distributed like the adjacency matrix of the
configuration model~$\CM$ introduced
in the proof of \cref{thm:moment-method}: then,
we may again condition on~$\CM$ being simple
($\CMA(i,i)=0$ and $\CMA(i,j)\le1$ for all $i\neq j\in[n]$)
to deduce
the star-moment
convergence,
for every word~$w$ on~$\Sigma$,
\[\frac1n\EE{\Tr{\RGA>^w}}
=\frac1n\EE\left[\Tr{\CMA^w}\mid\text{$\CM$ is simple}\right]
\conv
\varphi(a_d^w),\]
of the uniform $d$-regular digraph~$\RG>$
with adjacency matrix~$\RGA>$.

Finally, we check that the star moments~$\varphi(a_d^w)$
coincide with the number~$\MM(w)$ of~$w$-paths in~$\RT>$.
Using~(b), it was derived in~\cite{nica1997diagonal}
that (for every $i\in[d]$)
\[R_{u_i^{\vphantom*},u_i^*}(z_1,z_2)=\sum_{k=1}^\infty
{(-1)}^{k-1}\,C_{k-1}\left[{(z_1z_2)}^k+{(z_2z_1)}^k\right].\]
By freeness, $R_{a,a^*}(z_1,z_2)=R_{u_1^{\vphantom*},u_1^*}(z_1,z_2)
+\cdots+R_{u_d^{\vphantom*},u_d^*}(z_1,z_2)=d\,R_{u_1^{\vphantom*},u_1^*}(z_1,z_2)$,
and the structure of this $R$-transform
shows that the free cumulants
$\kappa_\ell(a_d^{w_1},\ldots,a_d^{w_\ell})$
(which we recover from~\eqref{eq:r-transform})
vanish if $w:=w_1\cdots w_\ell$ is not alternating:
\[\kappa_\ell(a_d^{w_1},\ldots,a_d^{w_\ell})
=\begin{cases}
  d\,{(-1)}^{p-1}\,C_{p-1},&\text{if $w$ is alternating: $w={(1{*})}^p$ or $w={({*}1)}^p$},\\[.4em]
  0,&\text{otherwise}.
\end{cases}\]
Recalling the definition of~$\ANC(w)$,
the moment-cumulant formula~\eqref{eq:moment-cumulant}
then easily yields
\[\varphi(a_d^w)
\:\:\:=\sum_{\pi\in\ANC(w)}
      \left(\prod_{V\in\pi}{(-1)}^{\frac{|V|}2-1}\,%
      C_{\frac{|V|}2-1}\right)d^{|\pi|},
\]
as in \cref{thm:starmoments}.

\subsection{The oriented Kesten--McKay conjecture.}
\cref{thm:moment-method}
states that
the uniform $d$-regular digraph~$\RG>$
converges in star moments
to the $d$-regular directed tree~$\RT>$.
As we saw in the previous section,
the star moments of~$\RT>$
agree with those of
the sum~$a_d:=u_1+\cdots+u_d$ of~$d$ free Haar unitary
elements
in some non-commutative probability space~$(\fA,\varphi)$.
This implies the convergence of mean empirical
\emph{singular value} distributions:
for every $z\in\CC$ and every
continuous bounded function~$f$,
\begin{align}
  \frac1n\EE{\Tr{f\left(\sqrt{{\bigl(\RGA>-z\mathrm I_n\bigr)}^*\bigl(\RGA>-z\mathrm I_n\bigr)}\right)}}
  &\conv\varphi\left(f\biggl(\sqrt{{(a_d-z1)}^*(a_d-z1)}\biggr)\right),\notag
\intertext{that is,}
  \int f(t)\,\mu_{|\RGA>-z|}(\dd t)
  &\conv\int f(t)\,\mu_{|a_d-z|}(\dd t),\label{eq:esvd-conv}
\end{align}
where~$X-z$ means that we subtract~$z$ times
the identity element to~$X$, and
$\mu_{|X|}$ is the spectral measure
of the positive operator
$|X|:=\sqrt{XX^*}$
(i.e., $\mu_{|X|}$ is the unique real probability measure
having the same moments as~$|X|$,
as given by the Riesz--Markov--Kakutani theorem).

Although~$X\in\{\RGA>,a_d\}$ is not a normal element,
there still exists~\cite{haagerup2000brown} a
unique probability
measure~$\mu_X$
(on~$\CC$),
known as the Brown measure of~$X$,
such that
\[\int\log{|z-\lambda|}\,\mu_X(\dd\lambda)=
\int\log(t)\,\mu_{|X-z|}(\dd t)\]
for every $z\in\CC$.
When $X=\RGA>$, $\mu_X$ is nothing but the
ESD $\frac1n\sum_{i=1}^n\delta_{\lambda_i(\RGA>)}$
of~$\RG>$. Since the star moments of~$X$
determine $(\mu_{|X-z|})_{z\in\CC}$
and thus~$\mu_X$,
and the star moments of~$a_d$ and~$\RT>$
coincide, we can also view~$\mu_{a_d}$
as the spectral measure of~$\RT>$.
However, we cannot directly use~\eqref{eq:esvd-conv}
to show
\begin{equation}
  \frac1n\EE{\sum_{i=1}^nf\bigl(\lambda_i(\RGA>)\bigr)}
\conv\int_{\CC}f(z)\,\mu_{a_d}(\dd z)
\label{eq:okm-conjecture}
\end{equation}
because the logarithm is not a bounded function.
There still lacks a uniform control on the smallest
singular value of~$\RGA>-z$ to validate
the oriented Kesten--McKay conjecture~\eqref{eq:okm-conjecture},
see \cite[Lemma~4.3]{bordenave2012around}.

%\printbibliography
%\bibliographystyle{amsplain}
%\bibliography{references}
\begingroup
\providecommand{\bysame}{\leavevmode\hbox to3em{\hrulefill}\thinspace}
\renewcommand{\MR}[1]{\relax\ifhmode\unskip\space\fi{\scshape\href{%
http://www.ams.org/mathscinet-getitem?mr=#1}{mr\,\text{\tiny#1}}}}
\providecommand{\MRhref}[2]{%
  \href{http://www.ams.org/mathscinet-getitem?mr=#1}{#2}
}
\providecommand{\href}[2]{#2}

\endgroup

\end{document}